\documentclass[11pt,reqno]{amsart}

 \usepackage{amsmath,amsthm,amsfonts}
\usepackage{latexsym,mathabx,txfonts}
\usepackage{amssymb}

\newcommand{\M}{\mathcal{M}}

\newcommand{\al}{\alpha}

\newcommand{\fy}{\varphi}

\newcommand{\p}{\partial}
\newcommand{\I}{\infty}

\newcommand{\R}{\mathbb{R}}

\renewcommand{\hat}{\widehat}

\def\la{\langle}
\def\ra{\rangle}
\def\calL{{\mathcal L}}

\numberwithin{equation}{section}

\newtheorem{thm}{Theorem}[section]
\newtheorem{cor}[thm]{Corollary}
\newtheorem{lem}[thm]{Lemma}
\newtheorem{prop}[thm]{Proposition}

\theoremstyle{remark}

\def\il{\int\limits}
\def\less{\lesssim}

\newcommand{\EQ}[1]{\begin{equation}  \begin{split} #1 \end{split} \end{equation} }
\newcommand{\Del}[1]{}

\newcommand{\pr}{\\ &}

\newcommand{\de}{\delta}

\newcommand{\supp}{\operatorname{supp}}

\def\pr{\partial}
\def\dist{\mathrm{dist}}
\def\calH{{\mathcal H}}

\def\nn{\nonumber}

\def\eps{\varepsilon}

\def\f{\frac}
\def\calS{\mathcal{S}}
\def\til{\tilde}
\def\fy{\varphi}

\begin{document}

\title[Critical threshold phenomenon]{Threshold phenomenon for the quintic wave equation in three dimensions}

\author{Joachim Krieger, Kenji Nakanishi, Wilhelm Schlag}



\thanks{Support of the National Science Foundation  DMS-0617854, DMS-1160817 for the third author, and  the Swiss National Fund for
the first author  are gratefully acknowledged. The latter would like to thank the University of Chicago for its hospitality in August 2012}

\begin{abstract}
For the critical focusing wave equation $\Box u = u^5$ on $\R^{3+1}$ in the radial case, we establish the role of the ``center stable'' manifold $\Sigma$ constructed in \cite{KS} near
the ground state $(W,0)$ as a threshold between  blowup and scattering to zero, establishing a conjecture going back to numerical work by Bizo\'n, Chmaj, Tabor~\cite{Biz}. The underlying topology is stronger than the energy norm. 
\end{abstract}

\maketitle 

\section{Introduction}

We consider the energy-critical focusing nonlinear wave equation 
\begin{equation}\label{eqn:foccrit}
\Box u = u^5,\,\Box  = \partial_t^2 - \Delta_x,\,u[0] = (u, u_t)_{t=0} = (u_0, u_1)
\end{equation}
on the Minkowski space $\R^{3+1}$ with radial data. The conserved energy is 
\[
E(u,\dot u)= \il_{\R^{3}} \big(\frac12 |\nabla_{t,x} u|^{2} - \f16 |u|^{6}\big)\, dx
\]
 In a remarkable series of papers, \cite{DKM1, DKM2, DKM3, DKM4} 
Duyckaerts, Kenig, and Merle gave the following characterization of the long-time dynamics for radial data $u[0]\in\dot H^{1}\times L^{2}(\R^{3})$ of arbitrary
energy: either one has type-I blowup, i.e., $\| u[t]\|_{\dot H^{1}\times L^{2}}\to\I$ in finite time, or the solution
decomposes into a (possible empty) sum of time-dependent dilates of the ground state stationary solution
\[
W(x): = (1+{|x|^2}/{3})^{-\frac{1}{2}}
\]
together with a radiation term that acts like a free wave, up to a $o(1)$ as $t\to T_{*}\in (0,\infty]$. Here $[0, T_{*})$ is the existence interval of the solution. See~\cite{DKM4}
for the precise theorem. We remark that Kenig, Merle~\cite{KM2} had studied the case of energies $ E(u_{0},u_{1})<E(W,0)$ and established a finite-time blowup vs.\ scattering dichotomy depending on whether $\|\nabla u_{0}\|_{2}>\|\nabla W\|_{2}$ or $\|\nabla u_{0}\|_{2}<\|\nabla W\|_{2}$. For the subcritical case, Payne and Sattinger~\cite{PS} had given
such a criterion but with global existence, and the scattering remained unknown. The latter gap was closed only recently by Ibrahim, Masmoudi, and the second author~\cite{IMN} using the Kenig-Merle method. 

The dynamics for the case $E(u_{0},u_{1})=E(W,0)$ was described by Duyckaerts, Merle~\cite{DM1, DM2} who constructed  the
one-dimensional stable and unstable manifolds associated with~$W$. Finally, \cite{DKM1} allowed energies
slightly larger than~$E(W,0)$, and it was shown there that general type-II blowup occurs by dynamical non-selfsimilar
rescaling of~$W$. The existence of such blowup solutions was established by the first and third authors and Tataru in~\cite{KST}. 
An analogous construction in infinite time was carried out by Donninger and the first author
 in~\cite{DoKr}.   In this context we would also like to mention the type-II blowup construction by Hillairet and Rapha\"{e}l~\cite{HR} 
 for the $4$-dimensional  semilinear wave equation. 

From a different perspective, and motivated in part by the phenomenological work~\cite{Biz} of Bizo\'n, Chmaj, and Tabor, the first and third authors investigated in~\cite{KS}
the question of {\em conditional stability} of the ground state~$W$. This is a very delicate question, and remains unsolved in the energy topology.
Note that the aforementioned blowup solutions can be chosen to lie arbitrarily close relative to the energy topology to the soliton curve $\calS:=\{W_{\lambda}\}_{\lambda>0}$ where
$W_{\lambda}(x)=\sqrt{\lambda} W(\lambda x)$. 
However, in a much stronger topology, \cite{KS} established the existence of a codimension-$1$ Lipschitz manifold~$\Sigma$ near $W$ so that data chosen
from this manifold exhibit asymptotically stable dynamics. See~\cite{KS} for the exact formulation. 

The question remained as to the dynamics for data near~$\Sigma$, but which do not fall on~$\Sigma$. As a start in this direction we mention the work by Karageorgis-Strauss \cite{KaraStr} for a related model equation of the same scaling class as \eqref{eqn:foccrit} where they show blow up for certain data with energy above that of the ground state, which are in a sense 'above the tangent space' of $\Sigma$. \\
In the subcritical case, the second and third
authors had shown, see~\cite{NakS0, NakS3, NakS2, NakS}, that this hypersurface~$\Sigma$ divides a small ball into two halves  which exhibit the finite-time blowup vs.\ scattering
dichotomy in forward time. This was carried out in the energy class, and $\Sigma$ was identified with the center-stable manifold associated
with the hyperbolic dynamics generated by linearizing about the ground state. See the seminal work by Bates, Jones~\cite{BJ} for an invariant manifold theorem
in infinite dimensions, with applications to a certain class of Klein-Gordon equations. \\

For the energy critical wave equation~\eqref{eqn:foccrit}, the authors~\cite{CNLW1, CNLW2} had shown a somewhat weaker result, 
namely the existence of four pairwise disjoint 
sets $A_{\pm,\pm}$ in the energy space near the soliton curve such that:  (i) each set has nonempty interior (ii)  the long-term dynamics (in both positive
and negative times) for data taken from each set
is determined as either blowup or global existence and scattering. 

However, the question of existence of a center-stable manifold near $W$ in the energy space remains open and appears delicate. 
Therefore, the results of~\cite{CNLW1, CNLW2} are not as complete as  those in~\cite{NakS}, in the sense that no comprehensive description of the dynamics near the soliton curve
is obtained. This is also explained by the fact that the dynamics of the energy critical equation appear more complex due to the scaling invariance which is not a feature of the Klein-Gordon equation considered in~\cite{NakS}, as evidenced by the variety of exotic type-II solutions. Moreover, the construction of the ``center-stable" manifold\footnote{We place ``center-stable'' in quotation marks, since $\Sigma$
cannot be interpreted as such an object. In fact,  the space $X_{R}$ is not invariant under the flow.} in~\cite{KS} is significantly more involved than the corresponding manifold for the subcritical Klein-Gordon equation.   

In this paper, we return to the point of view of~\cite{KS} in order to establish a description of all possible dynamics with data near $(W, 0)$ in the following main theorem, albeit in a stronger topology than that given by the energy. 
To formulate it, we need the linearized operator $H:= -\Delta-5W^{4}$. It exhibits a unique negative eigenvalue $-k_{0}^{2}$ with $Hg_{0}=-k_{0}^{2}\,g_{0}$,
and $g_{0}>0$ is smooth, radial, and exponentially decaying. 

\begin{thm}\label{thm:Main}
Fix $R>1$. There exists an $\eps_* = \eps_*(R)>0$ with the following property. Consider all initial data $(W+u)[0]: = (W + f_1, f_2)$ with $\|f_1\|_{H^3} + \|f_2\|_{H^2}<\eps_*$ and both $f_{1,2}$ supported within $B(0, R)$. Also, denote by $\Sigma$ the co-dimension one hypersurface within this neighborhood constructed in~\cite{KS}. Pick initial data $u[0]\in \Sigma$ with
\[
u(0, \cdot) = f_1 + h(f_1, f_2)g_0,\quad u_t(0, \cdot) = f_2
\]
where we have $\langle k_0 f_1 + f_2, g_0\rangle = 0$. Then the following holds: 
\begin{itemize}
\item if $\eps_*>\delta_0>0$, then initial data 
\[
\tilde{u}(0, \cdot) = W+ f_1 + (h(f_1, f_2)+\delta_0)g_0,\quad \tilde{u}_t(0, \cdot) = f_2
\]
lead to solutions blowing up in finite positive time.
\item if $-\eps_*<\delta_0<0$, then initial data 
\[
\tilde{u}(0, \cdot) = W+ f_1 + (h(f_1, f_2)+\delta_0)g_0,\quad \tilde{u}_t(0, \cdot) = f_2
\]
lead to solutions existing globally in forward time and scattering to zero in the energy space. 
\end{itemize}
\end{thm}

 The hyper-plane $\langle k_0 f_1 + f_2, g_0\rangle = 0$ is the tangent space to~$\Sigma$ at $(W,0)$, and it is denoted by~$\Sigma_{0}$ in~\cite{KS}.
The function $h$ is constructed in~\cite{KS} and for any $0<\delta\le \eps_{*}(R)$ one has the following properties: 
define the space
\[ X_R:=\{ (f_1,f_2)\in H^3_{\rm rad}(\R^3)\times H^2_{\rm rad}(\R^3)\:|\: \supp(f_j)\subset B(0,R)\} \]
Then 
$h:B_\delta(0)\subset\Sigma_0\to \R$ where $B_{\de}(0)$ is relative to~$X_{R}$ and 
one has  the estimates
\EQ{\nn 
|h(f_1,f_2)| &\less \|(f_1,f_2)\|_{X_R}^2, \qquad \forall\; (f_1,f_2)\in B_\delta(0)\nn \\
|h(f_1,f_2)-h(\til f_1, \til f_2)| &\less \delta \|(f_1,f_2)-(\til
f_1,\til f_2)\|_{X_R}\qquad \forall\; (f_1,f_2), (\til f_1,\til f_2)\in B_\delta(0) \label{eq:heta}
}
The Lipschitz graph $\Sigma$ is given by $(f_{1} + h(f_{1},f_{2})g_{0}, f_{2})$ where $(f_{1},f_{2})\in B_{\delta}(0)\subset \Sigma_{0}$. 
It is a Lipschitz hypersurface in~$X_{R}$ which approaches $\Sigma_{0}$ quadratically near the point~$(W,0)$. It is thus clear that $\Sigma_{0}$
is the tangent space to~$\Sigma$ at~$(W,0)$. 

Finally, we note that our choice of topology is not optimal for this type of theorem, and our approach can be extended to more general initial conditions.
On the other hand, we emphasize that the distinction between the energy topology $\dot H^{1}\times L^{2}$ on the one hand, and a stronger one such
as ours has very dramatic effects. Indeed, solutions starting on the manifold $\Sigma$ as constructed in~\cite{KS} are shown there to approach 
$
W_{a(\I)}
$
up to a radiation part where $a(\I)\in (0,\I)$. If a center-stable manifold can be constructed in $\dot H^{1}\times L^{2}$, then we cannot expect
the same behavior for solutions associated with such an object. Indeed, from \cite{KST} and~\cite{DoKr} we know that energy 
solutions exist arbitrarily close to $(W,0)$ in the energy topology for which $a(t)$ can approach either $0$ or~$\I$ in finite or infinite time.


The idea of the proof of the theorem is to combine the precise description of solutions with data on $\Sigma$ contained in~\cite[Definition 3]{KS} with the exit characterization of solutions established in~\cite{CNLW1}. The latter work allows us to confine ourselves to the situation in which the solution is close to ${\mathcal{S}}$, the family of rescalings $W_{\lambda} = \lambda^{\frac{1}{2}}W(\lambda\cdot)$ of $W$, whence we can rely purely on perturbative methods. 
The key for the proof is the following result.

\begin{prop}\label{prop:key} 
There exists $1\gg \eps_0\gg \eps_*$ with the following property: 
Let $\tilde{u}[0]$ be data as in Theorem~\ref{thm:Main}. Then there exist $\tilde{\delta}_0\neq 0$ of the same sign as $\delta_0$, a constant $k_\infty$ with $|k_0 - k_\infty|\ll 1$, and  a finite time $T = T(\tilde{u}[0])$ with $\eps_0= |\tilde{\delta}_0|e^{k_\infty T}\gg \eps_*$ and such that at time $t = T$, we have a decoupling 
\[
\tilde{u}(t, \cdot) = W_{\alpha_T} + \tilde{v}_{\alpha_T},\quad |1-\alpha_T|\ll 1, 
\]
with 
\begin{equation}\label{eqn:keyortho}
\langle  \tilde{v}_{\alpha_T}, \Lambda^*g_{\alpha_T}\rangle = 0,\quad \Lambda = r\partial_r +\frac{1}{2}
\end{equation}
and furthermore 
\begin{equation}\label{eqn:keygrowth}
\langle \tilde{v}_{\alpha_T},\,g_{\alpha_T}\rangle \simeq \tilde{\delta}_0e^{k_\infty T}
\end{equation}
\end{prop}

Proposition~\ref{prop:key} guarantees that data which are obtained by adding  $\delta_0 g_0$ to a point on~$\Sigma$ diverge exponentially away from~$\Sigma$. 
 The trajectory moves away from the ``tube'' of rescaled ground states ${\mathcal{S}}$ in a specific direction, depending on the sign of $\delta_0$.
Note that the ``excitation'' of the unstable mode~$g_{0}$ can be
arbitrarily small in Theorem~\ref{thm:Main}. This is the main distinction from our previous works~\cite{CNLW1, CNLW2}. Indeed, in those cases this excitation needed to be
sufficiently large so as to dominate the evolution from the beginning (and for as long as the trajectory remained inside a small neighborhood of~$(W,0)$, since
otherwise the linearized dynamics cannot be compared to the nonlinear one). 

At least on a heuristic level, our construction in Proposition~\ref{prop:key} is motivated by the generalizations of the well-known Hartman-Grobman linearization theorem 
which applies to ODEs of the form $\dot x=Ax+f(x)$ in~$\R^{n}$ where $f(0)=Df(0)=0$ provided $A$ has no eigenvalues on the imaginary axis. In that case
there exists a homeomorphism $y=y(x)$ near $x=0$ which linearizes the ODE in the sense that $\dot y=Ay$. If $A$ does have spectrum on
the imaginary axis, then there is a result known as Shoshitaishvili's theorem~\cite{Sho1, Sho2}, see also Palmer~\cite{Palm},   
which ensures partial linearization of the ODE in the form
\EQ{\label{lin}
\dot y = By + \fy(y),\quad \dot z = Cz, 
}
after a change of variables near $x=0$. Here $B$ has its spectrum on the imaginary axis, and $C$ is the hyperbolic part, and $\fy$ satisfies $\fy(0)=D\fy(0)=0$ (the $y$-equation captures the center-dynamics). Note that in the formulation~\eqref{lin} the center-stable manifold is precisely given by $z_{+}=0$ where $z_{+}$ are the coordinates for which $C$ is expanding. In addition, since the change of coordinates is in fact bi-H\"older  it also follows
from~\eqref{lin} that the center-stable manifold~$\M_{cs}$ is {\em exponentially repulsive} in the sense that if a trajectory starts near but not on $\M_{cs}$, then it will move away
exponentially from~$\M_{cs}$. 

However, in this paper we do not rely on a partial linearization as in~\eqref{lin} since such a result is not available in our context. 
Rather, we show that the coupling between the ``center-stable'' dynamics
obtained in~\cite{KS} and the unstable hyperbolic dynamics is of higher order in a suitable sense, which implies the exponential push
away from~$\Sigma$.

We conclude this introduction by showing how to deduce the main theorem from the previous proposition. 

\begin{proof}[Proof of Theorem~\ref{thm:Main} assuming Proposition~\ref{prop:key}] Picking $\eps_*$ sufficiently small, the theory of \cite{CNLW1} applies. In particular, while the data $\tilde{u}[0] = \big(\tilde{u}(0, \cdot), \tilde{u}_t(0, \cdot)\big)$ satisfy 
\EQ{\label{nahe}
\dist_{\dot{H}^1\times L^2}(\tilde{u}[0], \mathcal{S}\cup -\mathcal{S})\lesssim \eps_*
}
where we identify ${\mathcal{S}}: = (W_{\lambda}, 0)_{\lambda>0}$, we have 
\EQ{\label{fern}
\dist_{\dot{H}^1\times L^2}(\tilde{u}[T], \mathcal{S}\cup -\mathcal{S})\simeq  |\tilde{\delta}_0| e^{k_\infty T}
}
provided we choose $|\tilde{\delta}_0| e^{k_\infty T}$ (and thus $\eps_0$) sufficiently large in relation to $\eps_*$. Indeed, this is a direct consequence of~\eqref{eqn:keygrowth} combined with \cite[Lemma 2.2]{CNLW1}. But then equation (3.44) as well as Proposition 5.1, Proposition 6.2 in \cite{CNLW1} imply that data with $\delta_0>0$ result in finite time blow up, while data with $\delta_0<0$ scatter to zero as $t\rightarrow +\infty$, with finite Strichartz norms. 
\end{proof}

Inspection of this proof shows that we rely on several previous results. On the one hand, the proof of Proposition~\ref{prop:key} depends crucially
on the asymptotic analysis of the stable solutions constructed in~\cite{KS}, including all dispersive estimates of the radiative part. On the other hand, for the non-perturbative analysis we rely on key elements of our previous
work~\cite{CNLW1}, namely the one-pass theorem and the ejection mechanism in relation to the variational structure (see the $K$-functional in~\cite{CNLW1}).
Note also that the latter paper requires the main theorem from~\cite{DKM1} in order to preclude blowup in the regime $K\ge0$ once the solution has excited  the soliton tube.

\section{Proof of Proposition~\ref{prop:key}} 

It remains to prove Proposition~\ref{prop:key}, which we carry out via a bootstrap argument using suitable norms. The norms we use for the perturbation are adapted
from those introduced in~\cite{KS}. 

\subsection{A modified representation of the data}

Throughout we assume that $(f_1, f_2)$ satisfy the conditions of Theorem~\ref{thm:Main}. We start with data of the form $$(f_1+h(f_1, f_2)g_0, f_2)\in \Sigma$$ with the orthogonality condition $\langle k_0 f_1 +f_2, g_0\rangle  = 0$. According to~\cite{KS}, these data can be evolved globally in forward time to a function $u(t, \cdot)$ so that $W_{a(t)} + u(t, \cdot)$ solves~\eqref{eqn:foccrit}, with $|a(t) - a(0)|\ll 1$ for all $t\geq 0$. Let $g_{\infty} = g_{\infty}(f_1, f_2)$ be the unstable mode for the operator $$\calH(a(\infty)) = -\Delta -5W^4_{a(\infty)}=:-\Delta+V$$ which is the reference Hamiltonian at $t = +\infty$.  Writing $$\Sigma_{0}: = \{\langle k_0 f_1 +f_2, g_0\rangle  = 0\}$$ for
the tangent plane to~$\Sigma$, pick $\tilde{h}(f_1, f_2)$ such that 
\[
(f_1+h(f_1, f_2)g_0 - \tilde{h}(f_1, f_2)g_{\infty}, f_2)\in \Sigma_{0}.
\]
This is possible since  $\|g_0 - g_{\infty}\|_2\ll 1$. 
The map 
\[
(f_1, f_2)\mapsto (f_1+h(f_1, f_2)g_0 - \tilde{h}(f_1, f_2)g_{\infty}, f_2) = :(\tilde{f}_1, f_2)
\]
is a Lipschitz continuous\footnote{In fact, this map is smoother but we do not make this explicit in \cite{KS}.} homeomorphism from a small neighborhood $U\subset\Sigma_0$ of $0$ (within the admissible data set as in Theorem~\ref{thm:Main}) to another neighborhood~$V$. In fact, it equals the identity plus a Lipschitz map with very small Lip constant. This follows from the fact that (see \cite{KS}, Section~4)
\EQ{\nn
|\tilde{h}(f_1, f_2)|\simeq  |h(f_1, f_2)|&\lesssim \|(f_1, f_2)\|^2,\\
\big|h(f_1, f_2) - h(g_1, g_2)\big|&\ll \|f_1 - g_1\|_{H^3} +   \|f_2 - g_2\|_{H^2}
}
Committing abuse of notation, we write $\tilde{h} = \tilde{h}(\tilde{f}_1, f_2)$, $g_{\infty} = g_{\infty}(\tilde{f}_1, f_2)$, where it is to be kept in mind that $g_\infty$ is associated with the asymptotic  operator determined by the data $(f_1+h(f_1, f_2)g_0, f_2)$. Then we have the identity 
\[
f_1+h(f_1, f_2)g_0 = \tilde{f}_1 + \tilde{h}(\tilde{f}_1, f_2)g_{\infty}
\]
and furthermore 
\[
( \tilde{f}_1 + \tilde{h}(\tilde{f}_1, f_2)g_{\infty}, f_2) \in \Sigma
\]
We next need to find an analogous representation for the shifted initial data 
\[
(f_1+(h(f_1, f_2)+\delta_0)g_0, f_2)
\]
 Observe that the map 
 \[
 (\tilde{f}_1, f_2, \tilde{\delta}_0)\mapsto  \tilde{f}_1 + (\tilde{\delta}_0+\tilde{h}(\tilde{f}_1, f_2))g_{\infty}
\]
is again Lipschitz and a homeomorphism for small values of the arguments. In particular, we can write 
\[
f_1+(h(f_1, f_2)+\delta_0)g_0 = \tilde{f}_1 + (\tilde{\delta}_0+\tilde{h}(\tilde{f}_1, f_2))g_{\infty}
\]
where $\tilde{\delta}_0$ is a Lipschitz-function of $(f_1, f_2, \delta_0)$. Also, observe that $\Sigma$ divides the data space into two connected components, which can be characterized by 
$\tilde{\delta}_0>0$, $\tilde{\delta}_0<0$. The same comment applies to $\delta_0$, and necessarily $\delta_0>0$ corresponds to $\tilde{\delta_0}>0$. 

\subsection{The perturbative ansatz} Now given $f_1, f_2, \delta_0$, let $u$ be the solution of \eqref{eqn:foccrit} corresponding to the data 
\[
(W+\tilde{f}_1 + \tilde{h}(\tilde{f}_1, f_2)g_{\infty}, f_2),\,(\tilde{f}_1 + \tilde{h}(\tilde{f}_1, f_2)g_{\infty}, f_2)\in \Sigma
\]
These are of course in general different from $(f_1+ h(f_1, f_2)g_0, f_2)$. Note that $g_\infty$ is the unstable eigenmode corresponding to the evolution of $u$ at $t = +\infty$. Also, denote by 
$\tilde{u}$ the solution corresponding to the data 
\[
(W+f_1+(h(f_1, f_2)+\delta_0)g_0, f_2) = (W+\tilde{f}_1 + (\tilde{\delta}_0+\tilde{h}(\tilde{f}_1, f_2))g_{\infty}, f_2)
\]
We shall first make the simple perturbative ansatz
\begin{equation}\label{eqn:pert}
\tilde{u} = u + \eta = W_{a(t)} + u_{*} + \eta
\end{equation}
where we use the decoupling 
\[
u(t, \cdot) = W_{a(t)} + u_{*}(t, \cdot)
\]
given in \cite{KS} with the bounds 
\begin{equation}\label{eq:u_*bounds1}
\|u_*(t, \cdot)\|_{L_x^\infty}\leq \de \langle t\rangle^{-1},\quad \|\nabla_x u_{*}(t, \cdot)\|_{L_x^2+L_x^\infty}\leq  \de \langle t\rangle^{-\eps}
\end{equation}
\begin{equation}\label{eq:u_*bounds2}
\|\nabla u_*(t, \cdot)\|_{L_x^2} + \|\nabla^2 u_*(t, \cdot)\|_{L_x^2}\leq  \de,\quad |u_*(x, t)|\less \de \langle x\rangle^{-1}
\end{equation}
for suitable $\delta = \delta(\eps_*, R)\ll 1$; in fact, $\delta=C_0\eps_*$ where $C_0$ is a big constant (depending on~$R$).  For the dilation parameter one has the bounds 
\begin{equation}
\label{eq:abounds}
|a(t) - a_{\infty}|\leq \delta \langle t\rangle^{-1},\,|\dot{a}(t)|\leq \delta \langle t\rangle^{-2}
\end{equation}
and in particular $|a(t) - a_\infty|\ll 1$. 
In view of \eqref{eqn:pert}, we obtain the following equation for~$\eta$: 
\EQ{\label{eqn:eta}
\partial_{tt}\eta + \calH(a(\infty))\eta &= N(u_*+\eta, W_{a(t)}) - N(u_{*}, W_{a(t)}) \\
&\qquad+ \big(\calH(a(\infty)) - \calH(a(t))\big)\eta =: F(t)
 }
Here we set $\calH(a) = -\Delta_x - 5W_a^4$, and borrowing notation from~\cite{KS}, we have 
\EQ{\label{NL}
 N(v, W_{a}) &= (v+W_a)^5 - W_a^5 - 5W_a^4 v 
}
The right-hand side in~\eqref{eqn:eta} further equals
\EQ{\label{F}
F(t) &= 5(u^4-W_{a(t)}^4)\eta + 10 u^3\eta^2 + 10 u^2 \eta^3 + 5 u \eta^4 + \eta^5 \\
&\qquad\qquad\qquad + 5(W_{a(t)}^4-W_{a(\I)}^4)\eta \\
 (u^4-W_{a(t)}^4)\eta &= (u_*^4  +4 u_*^3 W_{a(t)} + 6 u_*^2 W_{a(t)}^2 + 4 u_* W_{a(t)}^3)\eta
}
Note that all terms linear in $\eta$ are of the form $o(\eta)$, and they are also localized in space due to the decay of $u_*$ and~$W$. 
 We shall write $ \calH(a(\infty)) = \calH_{\infty}$ from now on, and denote the corresponding unstable mode by $g_\infty$, with $\calH_\I g_\infty = -k_\infty^2 g_\infty$. 
 It is natural to decompose 
 \begin{equation}\label{etazerl}
 \eta = P_{g_\infty^{\perp}}\eta +\delta(t)g_\infty =: \tilde{\eta}(t, \cdot) + \delta(t)g_\infty
 \end{equation}
 The key to proving Proposition~\ref{prop:key}  is the following result.
 
 \begin{prop}\label{prop:bootstrap} Let $T>0$ be such that $|\tilde{\delta}_0| e^{k_\infty T}\le \eps_0$. Then for any $t\in [0, T]$, we have the bounds 
 \begin{equation}\label{eqn:bootstr}
|\delta(t)|\simeq  |\tilde{\delta}_0| e^{k_\infty t},\,\|\tilde{\eta}(t, \cdot)\|_{L_x^2}+ \|\nabla_x\tilde{\eta}(t, \cdot)\|_{L_x^2} + \|\nabla_x^2\tilde{\eta}(t, \cdot)\|_{L_x^2}\ll |\tilde{\delta}_0| e^{k_\infty t}
 \end{equation}
 for some fixed large $M$.  Also, $\delta(t)$ has the same sign as $\tilde{\delta}_0$. 
 \end{prop}
 \begin{proof}[Proof of Proposition~\ref{prop:bootstrap}] Recall that
 \[
 F(t) =  N(u_*+\eta, W_{a(t)}) - N(u_{*}, W_{a(t)}) + \big(\calH_\infty - \calH(a(t))\big)\eta
\]
Then according to Section~3 in \cite{KS}, we can write 
\EQ{\label{eq:delta}
\delta(t) &= (2k_\infty)^{-\frac{1}{2}}[n_+(t) + n_-(t)],\\
n_{\pm}(t)  &= (\frac{k_\infty}{2})^{\frac{1}{2}}\tilde{\delta}_0e^{\pm k_\infty t} + \int_0^t e^{\pm k_\infty(t-s)}\langle F(s), g_\infty\rangle\,ds
}
Moreover, we have the Duhamel-type formula
\begin{equation}\label{eq:tildeeta}
\tilde{\eta}(t, \cdot) = -\int_0^t\frac{\sin[(t-s)\sqrt{\calH_\infty}]}{\sqrt{\calH_\infty}}P_{g_\infty^{\perp}}F(s)\,ds
\end{equation}
Assume that the solution exists on some interval $[0, \tilde{T})$, $\tilde{T}\leq T$, and that it satisfies the following estimates, which we refer to as 
{\bf{bootstrap assumptions}}: 
 \EQ{\label{eqn:bootstrAss}
|\delta(t)| &\leq 10|\tilde{\delta}_0| e^{k_\infty t} \\
\|\tilde{\eta}(t, \cdot)\|_{L_x^2}+ \|\nabla_x\tilde{\eta}(t, \cdot)\|_{L_x^2} + \|\nabla_x^2\tilde{\eta}(t, \cdot)\|_{L_x^2} &\leq \frac{2}{K}|\tilde{\delta}_0| e^{k_\infty t}
}
 for some large $K$, which will be chosen to depend on $\eps_0$. 
 
 We shall now infer that $|\delta(t)|\simeq |\tilde{\delta}_0| e^{k_\infty t}$ with a proportionality factor in $[\frac{1}{4},4]$ and we will improve the second inequality by replacing $\frac{2}{K}$ by~$\frac{1}{K}$. A standard continuity argument then implies Proposition~\ref{prop:bootstrap}. 
 
 \bigskip
 
 {\it{{\bf (A)} Improving the bound on $\tilde{\eta}$.}} We start with the $L_x^2$-norm. To control it, we use the simple bound 
  \EQ{\label{eq:basic}
 \Big \|\frac{\sin(t\sqrt{\calH_{\infty}})}{\sqrt{\calH_{\infty}}}P_{g_{\infty}^{\perp}}f \Big\|_{L_x^2} &=  \Big \| \int_0^t  \cos(s\sqrt{\calH_{\infty}}) \, ds\, P_{g_{\infty}^{\perp}}   f \Big\|_{L_x^2} \\
 & \lesssim  |t|\|f\|_{L_x^2}
 }
 Assume that we have the bound 
\begin{equation}\label{eq:crudeF}
\|F(s, \cdot)\|_{L_x^2}\ll \frac{|\tilde{\delta}_0|}{K} e^{k_\infty s}
\end{equation}
Then  \eqref{eq:basic} implies
\begin{align*}
&\Big\|\int_0^t\frac{\sin[(t-s)\sqrt{\calH_\infty}]}{\sqrt{\calH_\infty}}P_{g_\infty^{\perp}}F(s)\,ds\Big\|_{L_x^2}\\&\ll \frac{|\tilde{\delta}_0|}{K}\int_0^t (t-s)e^{k_\infty s}ds
\less \frac{|\tilde{\delta}_0|}{K}e^{k_\infty t}
\end{align*}
which recovers the dispersive type bound for $\tilde{\eta}$. The above bound~\eqref{eq:crudeF} for $F$ can be easily proved: for the difference 
\[
 N(u_*+\eta, W_{a(t)}) - N(u_{*}, W_{a(t)})
 \]
 it suffices to consider the ``extreme'' terms 
 \begin{equation}\label{eq:list1}
 u_* W_{a(t)}^3\eta, \quad u_*^4 \eta,\quad  u^3\eta^2, \quad \eta^5,
 \end{equation}
 see~\eqref{F}.   
 We now check~\eqref{eq:crudeF} for each of these expressions, bounding $\eta$ as in~\eqref{etazerl} via~\eqref{eqn:bootstrAss} as follows:
 \[
 \|\eta(t, \cdot)\|_{L_x^2}+ \|\nabla_x{\eta}(t, \cdot)\|_{L_x^2} + \|\nabla_x^2{\eta}(t, \cdot)\|_{L_x^2} \le C_1 |\tilde{\delta}_0| e^{k_\infty t}
 \]
 with an absolute constant $C_1$. In what follows, we will need to ensure that $\eps_0\ll K^{-1}$ (so that also $\de\ll K^{-1}$). 
 
 \medskip
 
For the first term in \eqref{eq:list1}, we get 
\begin{align*}
\big\|u_* W_{a(t)}^3\eta\big\|_{L_x^2}\lesssim \|u_*\|_{L_x^\infty}\|W_{a(t)}^3\|_{L_x^\infty} \|\eta\|_{L_x^2}\ll \frac{|\tilde{\delta}_0|}{K}\langle t\rangle^{-1}e^{k_\infty t}
\end{align*}
For the second term in \eqref{eq:list1}, we get 
\begin{align*}
\big\|u_*^4 \eta\big\|_{L_x^2}\lesssim \|u_*\|_{L_x^\infty}^4\|\eta\|_{L_x^2}\ll \frac{|\tilde{\delta}_0|}{K}\langle t\rangle^{-4}e^{k_\infty t}
\end{align*}
For the third term in \eqref{eq:list1}, use that $H^2(\R^3)\subset L^\infty$ to obtain the bound 
\begin{align*}
\big\|u^3\eta^2\big\|_{L_x^2}\lesssim \|u^3\|_{L_x^\infty}\|\eta\|_{L_x^\infty}\|\eta\|_{L_x^2}\lesssim \tilde{\delta}_0^2e^{2k_\infty t}\ll \frac{|\tilde{\delta}_0|}{K}e^{k_\infty t}
\end{align*}
For the last term in \eqref{eq:list1}, we similarly obtain 
\begin{align*}
\big\|\eta^5\big\|_{L_x^2}\lesssim \|\eta\|_{L_x^\infty}^4\|\eta\|_{L_x^2}\lesssim |\tilde{\delta}_0|^5e^{5k_\infty t}\ll \frac{|\tilde{\delta}_0|}{K}e^{k_\infty t}
\end{align*}

In order to complete the proof of the bound \eqref{eq:crudeF}, it remains to control the term 
\[
\big(\calH_\infty - \calH(a(t))\big)\eta
\]
Due to the fast decay rate ($\simeq  \langle x\rangle^{-4}$) of the potential $V= -5W_{a(t)}^4$, one easily infers 
\[
\|\big(\calH_\infty - \calH(a(t))\big)\eta\|_{L_x^2}\lesssim |a(\infty) - a(t)||\tilde{\delta}_0|e^{k_\infty t}\ll \frac{|\tilde{\delta}_0|}{K}\langle t\rangle^{-1}e^{k_\infty t}
\]
This completes the bootstrap for the norm $\|\tilde{\eta}\|_{L_x^2}$. 

\bigskip

Next, consider the norm $\|\nabla\tilde{\eta}\|_{L_x^2}$. To control it, we use~\cite[eq.~(36)]{KS} with $V=-5W(a(\I))^4$: 
\begin{align*}
\|\nabla\tilde{\eta}\|_{L_x^2}&\leq \|\sqrt{\calH_\infty} \,\tilde{\eta}\|_{L_x^2} + \| |V|^{\frac{1}{2}}\tilde{\eta}\|_{L_x^2}\\
&\leq \int_0^t \|F(s, \cdot)\|_{L_x^2}\,ds + \||V|^{\frac{1}{2}}\|_{L_x^\infty}\|\tilde{\eta}\|_{L_x^2}\\
&\ll \frac{|\tilde{\delta}_0|}{K}e^{k_\infty t}
\end{align*}
Finally, we consider $\|\nabla_x^2\tilde{\eta}\|_{L_x^2}$:
\begin{align*}
\|\nabla^2 \tilde{\eta}\|_{L_x^2}&\leq \| \calH_\infty  \,\tilde{\eta}\|_{L_x^2} + \| V \tilde{\eta}\|_{L_x^2}\\
&\leq \int_0^t \| \sqrt{\calH_\I} P_{g_\I^\perp}F(s, \cdot)\|_{L_x^2}\,ds + \|V \|_{L_x^\infty}\|\tilde{\eta}\|_{L_x^2}
\end{align*}
The final term here is $\ll \frac{|\tilde{\delta}_0|}{K}e^{k_\infty t}$ as desired, and for the integral we continue using~\cite[eq.~(35)]{KS}:
\EQ{
\int_0^t \| \sqrt{\calH_\I} P_{g_\I^\perp}F(s, \cdot)\|_{L_x^2}\,ds \less \int_0^t \| \nabla F(s, \cdot)\|_{L_x^2}\,ds
}
To bound the integral on the right, we again consider the terms in \eqref{eq:list1}. For the first of these, we have 
\begin{align*}
&\big\|\nabla_x\big(u_{*}W^3_{a(t)}\eta\big)\big\|_{L_x^2}\\&\lesssim \|\nabla_x u_*\|_{L_x^2+L_x^\infty}\|W^3_{a(t)}\|_{L_x^\infty\cap L_x^{2}}\|\eta\|_{L_x^\infty} + \|u_*\|_{L_x^\infty}\|\nabla_x(W^3_{a(t)})\|_{L_x^\infty}\|\eta\|_{L_x^2}\\
&+\|u_*\|_{L_x^\infty}\|W^3_{a(t)}\|_{L_x^\infty}\|\nabla_x\eta\|_{L_x^2}\\
&\ll\frac{|\tilde{\delta}_0|}{K}\langle t\rangle^{-\eps}e^{k_\infty t} + 
\frac{|\tilde{\delta}_0|}{K}\langle t\rangle^{-1}e^{k_\infty t}
\end{align*}
For the second term in \eqref{eq:list1}, we obtain the contribution 
\begin{align*}
\big\|\nabla_x\big(u_*^4\eta\big)\big\|_{L_x^2}\lesssim \|\nabla_x u_*\|_{L_x^2}\|u_*^3\|_{L_x^\infty}\|\eta\|_{L_x^\infty} + \|u_*^4\|_{L_x^\infty}\|\nabla_x\eta\|_{L_x^2}\ll \frac{|\tilde{\delta}_0|}{K}\langle t\rangle^{-3}e^{k_\infty t}
\end{align*}
For the last two terms of \eqref{eq:list1}, we have the bounds 
\begin{align*}
\big\|\nabla_x\big(u^3\eta^2\big)\big\|_{L_x^2}&\lesssim \|\nabla_x(u^3)\|_{L_x^2}\|\eta\|_{L_x^\infty}^2 + \|u^3\|_{L_x^\infty}\|\nabla_x\eta\|_{L_x^2}\|\eta\|_{L_x^\infty}\\
&\lesssim \tilde{\delta}_0^2 e^{2k_\infty t}\ll \frac{|\tilde{\delta}_0|}{K}e^{k_\infty t}
\end{align*}
\begin{align*}
\big\|\nabla_x(\eta^5)\big\|_{L_x^2}\lesssim \|\nabla_x\eta\|_{L_x^2}\|\eta^4\|_{L_x^\infty}\lesssim |\tilde{\delta}_0|^5 e^{k_\infty t}\ll \frac{| \tilde{\delta}_0|}{K}e^{k_\infty t}
\end{align*}
Finally, one also easily checks that 
\begin{align*}
\big\|\nabla_x\big((\calH_\infty - \calH(a(t)))\eta\big)\big\|_{L_x^2}\lesssim |a(\infty) - a(t)||\tilde{\delta}_0|e^{k_\infty t}\ll \frac{|\tilde{\delta}_0|}{K}\langle t\rangle^{-1}e^{k_\infty t}
\end{align*}

Before continuing, we make the following important observation from the proof:
\begin{cor}\label{cor:betterbound} The bootstrap assumption implies that we can write for $j = 0,1,2$ 
\[
\nabla_x^j\tilde{\eta}(t) = \tilde{\eta}_1^{(j)} + \tilde{\eta}_2 ^{(j)}
\]
where we have 
\EQ{
&\|\tilde{\eta}_1^{(j)}(t, \cdot)\|_{L_x^2}\ll \frac{|\tilde{\delta}_0|}{K}\langle t\rangle^{-\eps}e^{k_\infty t}\\
&\|\tilde{\eta}_2^{(j)}(t, \cdot)\|_{L_x^2}\ll |\tilde{\delta}_0|^2e^{2k_\infty t}
}
\end{cor}

This corollary is important since it shows that the interactions of $\tilde{\eta}$ with itself as well as with the driving term $u_*$ are much weaker than the principal unstable component of~$\eta$, i.e., $\delta(t)$.  We will have to take advantage of this improved bound in order to control  the evolution of~$\de(t)$. 

\bigskip

{\it{{\bf (B)} Improving the control over $\delta(t)$.}}
In order to complete the bound on $\eta$, we next need to control the growth of the coefficients $n_{\pm}(t)$. This appears more difficult due to the quadratic interactions in $F(s, \cdot)$ of the form $u_{*}\eta W_{a(t)}^3$. The issue here is that the dispersive bound for $u_{*}$ only gives $\langle t\rangle^{-1}$ decay, which just fails to be integrable. 
\\
We start by deducing an improved bound for $n_-(t)$ departing from our bootstrap assumption. In view of~\eqref{eq:delta} we have  
\[
n_-(t) = (\frac{k_\infty}{2})^{\frac{1}{2}}\tilde{\delta}_0e^{-k_\infty t} + \int_0^t e^{-k_\infty (t-s)}\langle F(s, \cdot), g_\infty\rangle\,ds
\]
Using the bound \eqref{eq:crudeF} with the improvement implied by Corollary~\ref{cor:betterbound},  we get the bound 
\begin{equation}\label{eq:n-bound}
|n_-(t)|\less |\tilde{\delta}_0|\langle t\rangle^{-\frac{\eps}{2}}e^{k_\infty t} + \tilde{\delta}_0^2 e^{2k_\infty t}
\end{equation}
We now use this, together with Corollary~\ref{cor:betterbound} as well as the a priori bounds on $u_{*}$, to derive the improved control over~$n_+(t)$. We depart from the differential equation 
\EQ{\label{eq:n+eqn}
&\dot{n}_+(t) - k_\infty n_+(t) \\
& = \frac{n_+(t)}{(2k_\infty)^{\frac{1}{2}}}\langle g_\infty \big(20u_{*}W_{a(\infty)}^3 + (a(\infty) - a(t))\partial_{\lambda}V|_{\lambda = a(\infty)}\big), g_\infty\rangle    + F_+(t),
}
where we use the notation $V_{\lambda}: = -5W_{\lambda}^4$ and
\begin{align*}
 F_+(t) &= \frac{n_-(t)}{(2k_\infty)^{\frac{1}{2}}}\langle 20 u_{*}g_\infty W_{a(t)}^3, g_\infty \rangle + \langle 20 u_{*}\tilde{\eta}W_{a(t)}^3, g_\infty\rangle\\
 & \quad + 
 \frac{n_+(t)}{(2k_\infty)^{\frac{1}{2}}}\langle g_\infty\big(V_{a(\infty)} - V_{a(t)} -  (a(\infty) - a(t))\partial_{\lambda}V|_{\lambda = a(\infty)}\big), g_{\infty}\rangle \\
 &\quad 
 + G_+(t)
\end{align*}
with
\begin{align*}
G_+(t) &=  \frac{n_+(t)}{(2k_\infty)^{\frac{1}{2}}}\langle 20 u_{*}g_\infty ( W_{a(t)}^3 - W_{a(\infty)}^3), g_\infty \rangle \\
&\quad +\langle  N(u_*+\eta, W_{a(t)}) - N(u_{*}, W_{a(t)}) -  20\delta(t) u_{*}g_\infty W_{a(t)}^3, g_\infty\rangle\\
&\quad  + \langle \big(\calH_\infty - \calH(a(t))\big)[\tilde{\eta}+(2k_\infty)^{-\frac{1}{2}} {n_-(t)}g_{\infty}], g_\infty\rangle
\end{align*}
We infer from \eqref{eq:n+eqn} that 
\begin{equation}\label{eq:n+duhamel}\begin{split}
&n_+(t) = (\frac{\tilde{\delta}_0}{2})^{\frac{1}{2}}e^{k_\infty t +\Gamma(0,t)} + \int_0^t e^{k_\infty(t-s) + \Gamma(s,t)}F_+(s)\,ds
\end{split}\end{equation}
where we use the notation 
\[
\Gamma(s,t): = \int_s^t \langle g_\infty \big(u_{*}(s_1,\cdot) W_{a(\infty)}^3 + (a(\infty) - a(s_1))\partial_{\lambda}V|_{\lambda = a(\infty)}\big), g_\infty\rangle\,ds_1
\]
In order to proceed, we shall  obtain uniform bounds on the phase function~$\Gamma(s, t)$. These hinge   on Proposition~\ref{prop:fullextradisp}, to be proved in the next section. This proposition   implies that 
\EQ{\label{supst}
\sup_{s,t>0}\Big|\int_s^t \langle g_\infty \big(u_{*}(s_1,\cdot) W_{a(\infty)}^3, g_\infty\rangle\,ds_1\Big|\lesssim \|u_*\|_{L_x^\infty L_t^1}\ll 1
}
It remains to estimate 
\EQ{\label{aav}
\sup_{s,t} \int_s^t(a(\infty) - a(s_1))\,ds_1
}
Note that the integrand decays like $s_1^{-1}$ from the bounds in~\cite{KS}, which is no integrable. Lemma~\ref{lem:a(t)keybound}  shows nevertheless that~\eqref{aav} is
uniformly bounded. This again hinges on Proposition~\ref{prop:fullextradisp}. 

\begin{lem}\label{lem:a(t)keybound} We have the averaged estimate
\[
\sup_{t>0}\Big|\int_0^t(a(\infty) - a(s))\,ds\Big|\ll 1
\]
\end{lem}
\begin{proof}
Here we use the equation defining $a(t)$ in \cite{KS}, given by (51) in loc.\ cit., which we copy here for $t\gtrsim 1$: 
\begin{align*}
\dot{a}(t) = -c_0\big(\frac{a(t)}{a(\infty)}\big)^{\frac{5}{4}}\langle \partial_{\lambda}W_{\lambda}|_{\lambda = a(\infty)}, (V_{a(\infty)} - V_{a(t)})u_*(t, \cdot) + N(u_*(t, \cdot), W_{a(t)})\rangle
\end{align*}
We write this equation somewhat schematically in the form 
\begin{align*}
\dot{a}(t) &= -c_0(a(\infty) - a(t))\langle \partial_{\lambda}W_{\lambda}|_{\lambda = a(\infty)}, u_*(t, \cdot)\partial_{\lambda}V_{\lambda}|_{\lambda = a(\infty)}\rangle\\
&+ O(|a(\infty) - a(t)|^2\langle |\partial_{\lambda}W_{\lambda}|_{\lambda = a(\infty)}|, |u_*(t, \cdot)|\langle x\rangle^{-4}\rangle)\\
&- c_0\big(\frac{a(t)}{a(\infty)}\big)^{\frac{5}{4}}\langle \partial_{\lambda}W_{\lambda}|_{\lambda = a(\infty)}, N(u_*(t, \cdot), W_{a(t)})\rangle
\end{align*}
Set $\al(t):=a(\I)-a(t)$, and write this ODE in the form
\EQ{\label{aldot}
\dot\al &= -\al\,\sigma - H  \\
\sigma(t) &= -c_0 \langle \partial_{\lambda}W_{\lambda}|_{\lambda = a(\infty)}, u_*(t, \cdot)\partial_{\lambda}V_{\lambda}|_{\lambda = a(\infty)}\rangle \\
H(t): &= O(|a(\infty) - a(t)|^2\langle |\partial_{\lambda}W_{\lambda}|_{\lambda = a(\infty)}|, |u_*(t, \cdot)|\langle x\rangle^{-4}\rangle)\\
&- c_0\big(\frac{a(t)}{a(\infty)}\big)^{\frac{5}{4}}\langle \partial_{\lambda}W_{\lambda}|_{\lambda = a(\infty)}, N(u_*(t, \cdot), W_{a(t)})\rangle
}
Solving from $t=\I$ one obtains 
\EQ{\label{asol}
\al(t) = \int_t^\I e^{\int_t^s \sigma}\; H(s)\, ds
}
Proposition~\ref{prop:fullextradisp} implies that 
\[
\sup_{s,t}\Big| \int_t^s \sigma \Big| \ll 1
\]
which ensures that $e^{\int_t^s \sigma} = O(1)$ uniformly in~$s,t$. 
We now claim that 
\EQ{\label{intal}
\int_0^t (a(\infty) - a(\tilde{t}))\,d\tilde{t} &= t\int_t^\infty e^{ \int_t^s\sigma } \; H(s)\,ds \\
&\qquad+\int_0^t  s\sigma(s)  \int_s^\I  e^{ \int_s^{\tilde s} \sigma} \; H(\tilde s)\,d\tilde s       \,ds
}
To verify this, note first that both sides vanish at $t=0$. Furthermore, taking a derivative in~$t$ reduces the equation to~\eqref{asol}. 

One has the bound
\EQ{\label{Hbd}
|H(t)|\lesssim \langle u_*^2, \langle x\rangle^{-4}\rangle + \delta \langle t\rangle^{-3}
}
with $0<\delta\ll 1$.  Therefore,  on the one hand, 
\[
\sup_{t>0}\big|t\int_t^\infty  e^{ \int_t^s\sigma } \;
H(s)\,ds\big|  \ll 1
\]
On the other hand,  $\sup_{s\ge0} |s\sigma(s)|\ll 1$ whence 
\EQ{\label{Hinttransf}
\Big| \int_0^t  s\sigma(s)  \int_s^\I  e^{ \int_s^{\tilde s} \sigma} \; H(\tilde s)\,d\tilde s       \,ds \Big| &\less  \int_0^t \int_s^\I |H(\tilde s)|\, d\tilde s       \,ds \\
&  = t \int_t^\I |H(\tilde s)|\, d\tilde s + \int_0^t s  |H( s)|     \,ds
}
The first term is $\ll1$ from~\eqref{Hbd}, whereas the second integral is dominated by 
\[
\sup_{t>0}\Big |\int_0^t s |H(s)|\,ds\Big|\lesssim \sup_{s>0}\|s u_*(s, \cdot)\|_{L_x^\infty}\|u_{*}\|_{L_x^\infty L_s^1} +\delta\ll 1
\]
In conclusion~\eqref{Hinttransf} is $\ll1$ which completes the proof of the lemma. 
\end{proof}

In conjunction with~\eqref{Hinttransf} the lemma  implies that the phase corrections $\Gamma(s, t)$ are uniformly small. 

We next estimate the contributions of the various constituents of $F_+(s, \cdot)$ to the integral in~\eqref{eq:n+duhamel}. This will then lead to the completion of  the proof of Proposition~\ref{prop:bootstrap}. 

\medskip

{\it{{\bf (1)}  The contribution of $ \frac{n_-(t)}{(2k_\infty)^{\frac{1}{2}}}\langle u_{*}g_\infty W_{a(t)}^3, g_\infty\rangle + \langle u_{*}\tilde{\eta}W_{a(t)}^3, g_\infty\rangle$.}} 

Using \eqref{eq:n-bound} as well as Corollary~\ref{cor:betterbound}, we bound this by 
\begin{align*}
&\ll \int_0^t e^{k_\infty(t-s) +\Gamma(s,t)}\langle s\rangle^{-1}[|\tilde{\delta}_0|\langle s\rangle^{-\frac{\eps}{2}}e^{k_\infty s} + \tilde{\delta}_0^2 e^{2k_\infty s}]\,ds\\
&\lesssim |\tilde{\delta}_0|e^{k_\infty t} + \tilde{\delta}_0^2 e^{2k_\infty t}
\end{align*}

\medskip

{\it{{\bf (2)} The contribution of $ \frac{n_+(t)}{(2k_\infty)^{\frac{1}{2}}}\langle g_\infty\big(V_{a(\infty)} - V_{a(t)} -  (a(\infty) - a(t))\partial_{\lambda}V|_{\lambda = a(\infty)}\big), g_{\infty}\rangle$.}}

We can bound this by 
\begin{align*}
\lesssim |\tilde{\delta}_0|\int_0^t e^{k_\infty(t-s) +\Gamma(s,t)}e^{k_{\infty}s}|a(\infty) - a(s)|^2\,ds\ll |\tilde{\delta}_0|e^{k_{\infty}t}
\end{align*}

\bigskip 
We next consider the contributions of the constituents of $G_+(t)$:

\smallskip 

{\it{{\bf (3)} The contribution of $\frac{n_+(t)}{(2k_\infty)^{\frac{1}{2}}}\langle 20 u_{*}g_\infty (W_{a(t)}^3-W_{a(0)}^3), g_\infty\rangle$.}} 

Use the bound 
\[
\big|n_+(t)\langle u_{*}g_\infty (W_{a(t)}^3-W_{a(\infty)}^3), g_\infty\rangle \big|\ll \langle t\rangle^{-1}|a(\infty) - a(t)| |n_+(t)|
\]
Hence the corresponding contribution is bounded by 
\begin{align*}
&\ll \int_0^t e^{k_\infty(t-s) + \Gamma(s,t)} \langle s\rangle^{-1}|a(\infty) - a(s)| |n_+(s)|\,ds\\
&\ll \int_0^t e^{k_\infty (t-s) + \Gamma(s,t)}\langle s\rangle^{-2}|\tilde{\delta}_0| e^{k_\infty s}\,ds\lesssim |\tilde{\delta}_0|e^{k_\infty t}
\end{align*}
where we have used the bound \eqref{eq:n-bound} as well as the bootstrap assumption to control~$n_+(t)$. 

\medskip 

{\it{{\bf (4)}  The contribution of $\langle  N(u_*+\eta, W_{a(t)}) - N(u_{*}, W_{a(t)}) -  20\delta(t) u_{*}g_\infty W_{a(t)}^3, g_\infty\rangle$.}}

Here we need to estimate the contributions of the following schematically written terms: 
\begin{equation}\label{eq:list2}
\langle u_*\tilde{\eta}W_{a(t)}^3, g_\infty\rangle,\; \langle \eta^2W_{a(t)}^3, g_\infty\rangle, \; \langle \eta u_*^4, g_\infty\rangle, \; \langle \eta^5, g_\infty\rangle
\end{equation}
For the first term, we can bound the contribution by 
\begin{align*}
\ll \int_0^t e^{k_\infty(t-s) +\Gamma(s,t)}(\langle s\rangle^{-1-\frac{\eps}{2}}|\tilde{\delta}_0| e^{k_\infty s} + \tilde{\delta}_0^2 \langle s\rangle^{-1}e^{2k_\infty s})\,ds\lesssim|\tilde{\delta}_0| e^{k_\infty t}
\end{align*}
The remaining terms are handled similarly. 

\medskip

{\it{{\bf (5)}  The contribution of $ \langle \big(\calH(a(\infty)) - \calH(a(t))\big)[\tilde{\eta} + (2k_\infty)^{-\frac{1}{2}} {n_-(t)}g_\infty], g_\infty\rangle$. }} 

Using \eqref{eq:n-bound} and 
Corollary~\ref{cor:betterbound}, we bound the corresponding contribution by the exact same expression as in {\it{(4)}}.

\smallskip

This completes the proof of Proposition~\ref{prop:bootstrap}. 
 \end{proof}

\medskip

It remains to prove Proposition~\ref{prop:key}. Thus fix a time $T$ with $1\gg |\tilde{\delta}_0|e^{k_\infty T}\gg \eps_*$ where we can write 
\[
\tilde{u}(T, \cdot) = W_{a(T)} + u_{*} + \eta
\]
as before. We need to pass to a representation 
\begin{equation}\label{eq:decomp3}
\tilde{u}(T, \cdot)  = W_{\alpha_T} + \tilde{v}_{\alpha_T}
\end{equation}
which satisfies $\langle  \tilde{v}_{\alpha_T}, \Lambda^*g_{\alpha_T}\rangle = 0$. From \cite{KS} we can write 
\[
u_*(t, \cdot) = P_{g_{\infty}^{\perp}}u_{*} + \delta_*(t)g_{\infty},\,|\delta_*(t)|\lesssim C(\eps_*)\langle t\rangle^{-1}
\]
In order to obtain the desired decomposition \eqref{eq:decomp3}, we need to satisfy the relation 
\begin{equation}\label{eq:keyortho}
\langle P_{g_{\infty}^{\perp}}(u_{*} + \eta) + (\delta(T) + \delta_*(T))g_{\infty} + W_{a(T)} - W_{\alpha_T}, \Lambda^*g_{\alpha_T}\rangle = 0
\end{equation}
Observe that 
\[
W_{a(T)} - W_{\alpha_T} = (a(T) -\alpha_T)\partial_\lambda W_{\lambda}|_{\lambda = a(T)} + O(|a(T) -\alpha_T|^2)
\]
and from (2.13) in \cite{CNLW1} we have 
\[
|\langle \partial_\lambda W_{\lambda}|_{\lambda = a(T)}, \Lambda^*g_{a(T)}\rangle|\simeq  1
\]
It follows that for $|a(T) -\alpha_T|\ll 1$ there is a unique solution of \eqref{eq:keyortho} which satisfies 
\[
|a(T) -\alpha_T|\lesssim |\tilde{\delta}_0| e^{k_\infty T}\ll 1
\]
To verify the condition \eqref{eqn:keygrowth}, we need to compute 
\begin{equation}
\langle  P_{g_{\infty}^{\perp}}(u_{*} + \eta) + (\delta(T) + \delta_*(T))g_{\infty} + W_{a(T)} - W_{\alpha_T}, g_{\alpha_T}\rangle
\end{equation}
From Proposition~\ref{prop:bootstrap} we have 
\[
|\delta(T)|\gg |\langle  P_{g_{\infty}^{\perp}}(u_{*} + \eta), g_{\alpha_T}\rangle| + |\delta_*(T)|,
\]
and furthermore
\[
|\langle W_{a(T)} - W_{\alpha_T}, g_{\alpha_T}\rangle| = O(|a(T) -\alpha_T|^2)\ll |\tilde{\delta}_0| e^{k_\infty T}\simeq  |\delta(T)|
\]
We have now proved the key growth condition 
\[
\langle \tilde{v}_{\alpha_T}, g_{\alpha_T}\rangle \simeq  \tilde{\delta}_0 e^{k_\infty T}
\]
which completes the proof of Proposition~\ref{prop:key}.

\section{Proof of the  dispersive estimate on $\|u_{*}\|_{L_x^\infty L_t^1}$.}

This section is devoted to the one estimate, namely on $\|u_{*}\|_{L_x^\infty L_t^1}$, which is not contained in~\cite{KS}. 
As evidenced by the previous section this norm is of crucial importance for the nonlinear argument.  

This section is devoted to the proof of this estimate, starting with the linear case. 
We use the expansions for the linear evolution associated with $\Box +V$, $V = -5W^4$, as derived in~\cite{KS}. 
In what follows, $H=-\Delta+V$ in $\R^3$ where $H\psi=0$ and $\psi$ is the unique zero energy resonance function, i.e., $|\psi(x)|\simeq |x|^{-1}$
for large~$|x|$.  We assume that $H$ does not have zero energy eigenfunctions. 

\begin{prop}\label{prop:extradisp}
We have the bounds 
\begin{align}
&\Big\|\Big(\frac{\sin(t\sqrt{H})}{\sqrt{H}}P_c-c_0\psi\otimes\psi
\Big)f\Big\|_{L_x^\infty L_t^1}\lesssim \|f\|_{{W}^{1,1}}  \label{sin_xt} \\
&\Big\|\cos(t\sqrt{H})P_cf\Big\|_{L_x^\infty L_t^1}\lesssim  \|f\|_{{W}^{2,1}}  \label{cos_xt}
\end{align}
\end{prop}
\begin{proof} We begin with $V = 0$. For the sine evolution, we get (putting the argument $x = 0$)
\begin{align}
&\int_0^\infty\frac{1}{t}\Big|\il_{[|y|=t]} f(y)\,\sigma(dy) \Big|\,dt
 = \int_0^\infty t^{-2}\Big|\il_{[|y|\le
t]} \nabla(f(y)y)\, dy\Big|\,dt \label{eq:strauss}\\
&\less \il_{\R^3} |\nabla f(y)|\, dy + \Big(\il_{\R^3}
\frac{|f(y)|}{|y|}\,
dy\Big) \less \il_{\R^3} |\nabla f(y)|\, dy\nn
\end{align}
The last step uses integration by parts in polar coordinates. 

For the cosine evolution, one has 

\begin{align*}
\cos(t\sqrt{H}) f(x) &= \pr_t\; t \il_{S^2} f(x+ty)\, \sigma(dy) \\
& = \il_{S^2} \big[ f(x+ty)+ t(\nabla f)(x+ty)\cdot y\big]\,
\sigma(dy) \\
\end{align*}
and so 
\[
\big\|\cos(t\sqrt{H}) f\big\|_{L_x^\infty L_t^1}\less \big \|\frac{f}{|x|^2}  \big\|_{L_x^1} +   \big\|\frac{\nabla f}{|x|} \big\|_{L_x^1} \less \|D^2 f\|_{L_x^1}
\]

In case $V\ne0$ we write the $\frac{\sin(t\sqrt{H})}{\sqrt{H}}$ evolution  in the form 
\EQ{\label{eq:sintrans} \f{1}{i\pi}\il_0^\infty
\frac{\sin(t\lambda)}{\lambda}
[R_V^+(\lambda^2)-R_V^{-}(\lambda^2)]\,\lambda d\lambda =
\f{1}{i\pi} \il_{-\infty}^\infty \sin(t\lambda) R(\lambda) d\lambda
} where we have set $R(\lambda):=R_V^+(\lambda^2)$ if
$\lambda>0$ and $R(\lambda)=\overline{R(-\lambda)}$ if
$\lambda<0$. For the free resolvent, we write this as $R_0$. Then,
by the usual resolvent expansions,
\EQ{  \label{eq:res} R=
\sum_{k=0}^{2n-1} (-1)^k R_0(VR_0)^k + (R_0V)^n R (VR_0)^n 
}
We distinguish between small energies and all other energies. For
the latter, we use~\eqref{eq:res}. Let $\chi_0(\lambda)=0$ for all
$|\lambda|\le \lambda_0$ and $\chi_0(\lambda)=1$ if
$|\lambda|>2\lambda_0$. Here $\lambda_0>0$ is some small parameter.
Fix some $k$ as in~\eqref{eq:res} and consider the contribution of
the corresponding Born term (ignoring a factor of $(4\pi)^{-k-1}$):
\EQ{\label{born} 
&\il_{\R^{3(k+2)}} \il_{-\infty}^\infty
\chi_0(\lambda) \sin(t\lambda) e^{i\lambda\sum_{j=0}^{k}
|x_j-x_{j+1}|}
 \frac{\prod_{j=1}^k V(x_j)}{\prod_{j=0}^k|x_j-x_{j+1}|}
 \,f(x_0)\,d\lambda\,dx_0\ldots dx_{k} \\ &
=\f{1}{2i}\sum_{\pm}\pm \int_{\R}\il_{\R^{3k}} \widehat{\chi_0}(\xi)
\il_{[|x_0-x_1|=\pm t-\xi-\sum_{j=1}^{k} |x_j-x_{j+1}|>0]}
\frac{f(x_0)}{|x_0-x_1|}\, \sigma(dx_0)\,
 \\ &
 \qquad \qquad \frac{\prod_{j=1}^k V(x_j)}{\prod_{j=1}^k|x_j-x_{j+1}|}
 \,dx_1\ldots dx_{k}\, d\xi
}
where $x_{k+1}$ is fixed.  Placing absolute values  inside these integrals and
  integrating over $t\in\R$ yields  an upper bound
  \EQ{
& \int_{\R}  |\widehat{\chi_0}(\xi)|  \il_{\R^{3}}  \frac{|f(x_0)|}{|x_0-x_1|}\, dx_0\,
 \il_{\R^{3k}}  \frac{\prod_{j=1}^k |V(x_j)| }{\prod_{j=1}^k|x_j-x_{j+1}|}
 \,dx_1\ldots dx_{k}\, d\xi \\
 &\less  \| \nabla f \|_{1}  \| \nabla V \|_{1}^{k}
  }
It remains to bound the contribution by the final term
in~\eqref{eq:res} which involves the resolvent~$R(\lambda)$. Its
 kernel $K(x,y)$ can be reduced to the form
\begin{align}
&\int e^{\pm it\lambda}\chi_0(\lambda) \la R(\lambda)
(VR_0(\lambda))^n(\cdot,x),
(VR_0(-\lambda))^n(\cdot,y) \ra \,d\lambda \nn\\
&= \int e^{i\lambda[\pm t+(|x|+|y|)]}\chi_0(\lambda) \la R(\lambda)
(VR_0(\lambda))^{n-1}VG_x(\lambda,\cdot), \\
&\qquad 
(VR_0(-\lambda))^{n-1}VG_y(-\lambda,\cdot) \ra \,d\lambda
\label{eq:Gint}
\end{align}
where
\[ G_{x}(\lambda,u) := \frac{e^{i\lambda(|x-u|-|x|)}}{4\pi|x-u|} \]
and the scalar product appearing in~\eqref{eq:Gint} is just another way of writing the composition of
the operators.
One has the following elementary bounds, see for example Lemma~11 in \cite{KS}: 
\EQ{
\label{eq:Gest}
\sup_{x\in\R^3} \Big\| \frac{d^j}{d\lambda^j} G_{x}(\lambda,
\cdot)\Big\|_{L^{2,-\sigma}} & < C_{j,\sigma} \text{\ \ provided\ \ } \sigma > \frac12 + j  \\
\sup_{x\in\R^3} \Big\| \frac{d^j}{d\lambda^j}
G_{x}(\lambda,\cdot)\Big\|_{L^{2,-\sigma}} &<
\frac{C_{j,\sigma}}{\la x\ra} \text{\ \ provided\ \ }
\sigma>\frac32+j
}
for all $j\ge 0$.    Let for some large $n$ (say $n=10$)
\[
a_{x,y}(\lambda) := \chi_0(\lambda) \la R(\lambda)
(VR_0(\lambda))^{n-1}VG_x(\lambda,\cdot),
(VR_0(-\lambda))^{n-1}VG_y(-\lambda,\cdot) \ra
\]
Then in view of the preceding one concludes that $a_{x,y}(\lambda)$
has two derivatives in~$\lambda$ and 
\begin{equation}
\label{eq:adec}
\Big|\frac{d^j}{d\lambda^j} a_{x,y}(\lambda)\Big| \less
(1+\lambda)^{-2} 
\text{\ \ for\ \ } j = 0,1,2 \text{\ \ and all\ \ }\lambda>1 
\end{equation}
Moreover, 
\EQ{\label{a2}
\Big|\frac{d^j}{d\lambda^j} a_{x,y}(\lambda)\Big| \less
(1+\lambda)^{-2} (\la x\ra\la y\ra)^{-1}
\text{\ \ for\ \ } j = 0,1, \text{\ \ and all\ \ }\lambda>1
}
The decay in $\lambda$ here comes  from the limiting absorption principle
which refers to the following standard bounds for the free and
perturbed resolvents:  
\begin{align}
\|R_V(\lambda^2\pm i0)\|_{L^{2,\sigma}\to L^{2,-\sigma}} &\less
\lambda^{-1}, \qquad \sigma>\f12 \label{eq:limap}\\
\|\pr_\lambda^\ell R_V(\lambda^2\pm i0)\|_{L^{2,\sigma}\to
L^{2,-\sigma}} &\less 1, \qquad \sigma>\f12+\ell, \qquad \ell\ge1 \nn
\end{align}
for $\lambda$ separated from zero. The estimates~\eqref{eq:adec} and~\eqref{a2} only
require $|V(x)|\less \la x\ra^{-\kappa}$ with $\kappa>3$.

Let us assume first that $t>1$. To estimate~\eqref{eq:Gint} we
distinguish between $|t-(|x|+|y|)|<t/10$ and the opposite case.  In
the former case, we conclude that
\[ \max(|x|,|y|)\gtrsim t \]
so that due to \eqref{eq:adec} we obtain
\EQ{\label{bd1} \Big| \int e^{i\lambda[\pm t+(|x|+|y|)]}
a_{x,y}(\lambda)\,d\lambda \Big| \less \chi_{[|x|+|y|>t]}(\la
x\ra\la y\ra)^{-1}
}
Integrating \eqref{bd1} over $t\in\R$ yields a bound $O(1)$ which implies an $L^{1}_{x}\to L^{\I}_{y}L^{1}_{t}$
estimate. 

In the latter case we integrate by parts twice which gains 
$t^{-2}$ for $|t|>1$ from~\eqref{eq:adec}:
\[ \Big| \int e^{i\lambda[\pm t+(|x|+|y|)]}
a_{x,y}(\lambda)\,d\lambda \Big| \less |t|^{-2}
\]
For $|t|\less1$ one has $O(1)$. We can again integrate this over $t\in\R$ as before. 

We now turn to the contribution of small $\lambda$ to the $\sin$-evolution.  We recall
the following representation of the resolvent at small energies, see (105) in \cite{KS}: 
\begin{align}
\label{eq:3piece}
 R(\lambda) &=i\f{\beta}{\lambda} R_0(\lambda)vS_1v
R_0(\lambda) +R_0 (\lambda) - R_0(\lambda)vE(\lambda)vR_0(\lambda).
\end{align}
where with $w:=\sqrt{|V|}$, 
\[
S_1 = \|w\psi\|_2^{-2} w\psi\otimes w\psi =: \tilde\psi\otimes \tilde\psi
\]
and $\beta=4\pi \Big(\il_{\R^3} V\psi\, dx\Big)^{-2}\|w\psi\|_2^2$. For the explicit form of $E(\lambda)$ see (104) in \cite{KS}. 
Next, we describe the contribution of each of the three terms in~\eqref{eq:3piece} to the
sine-transform~\eqref{eq:sintrans}. We can ignore the second one,
since it leads to the free case. The first term on the right-hand
side of~\eqref{eq:3piece} yields the following expression
in~\eqref{eq:sintrans}:
\EQ{
&\calS_0(t)(x,y) := \f{\beta}{\pi} \int \f{\sin(t\lambda)}{\lambda}
\chi_1(\lambda)
\big[R_0(\lambda) vS_1v R_0(\lambda)\big](x,y) \, d\lambda \\
& := \|w\psi\|_2^{-2}{\beta}\, \psi(x)\psi(y) - \|w\psi\|_2^{-2}\f{\beta}{2\pi} \il_{\R^6} \int_{[|\tau|>t]}
\widehat\chi_1(\tau+|x-x'|+|y'-y|) \nn \\
&  \qquad\qquad
\f{V(x')\psi(x')\,V(y')\psi(y')}{4\pi|x-x'|\, 4\pi|y'-y|}\,
d\tau\,dx'dy' \label{eq:VV}
}
We need to verify that uniformly in $x,y\in\R^{3}$ the integral over $t\in\R$ of the last
line is~$O(1)$. 
Indeed  
\begin{align*}
& \il_{\R^6}\il_{\R^6} \int_{[|\tau|>t]}
|\widehat\chi_1(\tau+|x-x'|+|y'-y|)|
\f{|V(x')\psi(x')\,V(y')\psi(y')|}{4\pi|x-x'|\, 4\pi|y'-y|}\,
d\tau\,dx'dy'\Big|\\
& \less \il_{[|x-x'|+|y-y'|<t/2]} \int_{[|\tau|>t]}
|\widehat\chi_1(\tau+|x-x'|+|y'-y|)|\, d\tau \\
&\qquad\qquad 
\f{|V(x')\psi(x')|\,|V(y')\psi(y')|}{|x-x'|\,
|y'-y|}\,dx'dy'  \\
&\quad +\il_{[|x-x'|+|y-y'|>t/2]}\int
|\widehat\chi_1(\tau+|x-x'|+|y'-y|)|\, d\tau \\
&\qquad\qquad 
\,\f{|V(x')\psi(x')|\,|V(y')\psi(y')|}{|x-x'|\, |y'-y|}\, dx'dy' 
\end{align*}
The first integral in the final expression is rapidly decaying in~$t$, and thus
gives the desired bound, whereas the second one upon integration in~$t$ is bounded by
\EQ{
&\il ( |x-x'|+|y-y'|)
\,\f{|V(x')\psi(x')|\,|V(y')\psi(y')|}{|x-x'|\, |y'-y|}\, dx'dy'   \less 1
}
Finally, we turn to the third term on the right-hand side
of~\eqref{eq:3piece}.  The convergence of the Neumann
series defining~$E(\lambda)$  in $L^{2}$ for small~$\lambda$ was established in \cite{KS}. 
We analyze the contribution by the constant term, viz.
\[ E(0)=(A_0+S_1)^{-1} + E_1(0)S_1m(0)^{-1}S_1+S_1E_2(0)S_1 + S_1 m(0)^{-1}S_1 E_1(0) \]
see (104) in \cite{KS}.  From (108), (109) in \cite{KS} one has 
\begin{align}
&  \il_{\R^{3}}\il_{-\infty}^\infty \sin(t\lambda) \chi_1(\lambda)
[R_0(\lambda)vE(0)vR_0(\lambda)](x,y)\, d\lambda \; f(x)\, dx   \nn\\
&=\f{1}{32 i\pi^2} \il_{\R^{3}} \il_{\R^6} \il_{-\infty}^\infty \delta( t + \xi +
[|x-x'|+|y'-y|])\;\widehat\chi_1(\xi) \, d\xi\\ & \qquad\qquad 
\f{v(x')E(0)(x',y')v(y')}{|x-x'|\,|y-y'|}\, dx'dy'  \, f(x)\, dx\nn     \\
& \quad -\f{1}{32 i\pi^2}  \il_{\R^{3}} \il_{\R^6} \il_{-\infty}^\infty \delta(- t
+ \xi + [|x-x'|+|y'-y|])\;\widehat\chi_1(\xi) \, d\xi\\ &\qquad\qquad 
\f{v(x')E(0)(x',y')v(y')}{|x-x'|\,|y-y'|}\, dx'dy'  \; f(x)\, dx \nn 
\end{align}
Placing absolute values inside these expressions and integrating over $t\in\R$ yields
an upper bound of the form (for $y$ fixed)
\EQ{
&\less  \il_{\R^{3}} \il_{\R^6} \il_{-\infty}^\infty | \widehat\chi_1(\xi)| \, d\xi\,
\f{|v(x')E(0)(x',y')v(y')|}{|x-x'|\,|y-y'|}\, dx'dy'  \, |f(x)| \, dx\nn     \\
& +  \il_{\R^{3}} \il_{\R^6} \il_{-\infty}^\infty | \widehat\chi_1(\xi) | \, d\xi\,
\f{|v(x')E(0)(x',y')v(y')|}{|x-x'|\,|y-y'|}\, dx'dy'  \; |f(x)|\, dx \nn 
}
which in turn is bounded by
\EQ{
\|  \widehat\chi_1 \|_{1} \sup_{x}\Big \|  \f{v(x')}{|x-x'|} \Big\|_{L^2_{x'}}^{2}   \|\,|E(0)(\cdot,\cdot)|\,\|_{2\to2} \|f\|_{1} 
\less \|f\|_{1}
}
since $E(0)$ is absolutely bounded on~$L^{2}$, see \cite{KS}. 

To deal with $E(\lambda)$ we proceed as in \cite{KS} using the $F(\lambda)$-method. To be specific, we claim the bound
\EQ{\label{Flam}
\il_{-\I}^{\I }\Big| \il_{\R^6} \il_{-\infty}^\infty \sin(t\lambda) \chi_1(\lambda)
[R_0(\lambda)vF(\lambda)vR_0(\lambda)](x,y)\, d\lambda\, f(x)\,
dx  \Big| \, dt \less  \| f\|_1 
}
provided the operator-valued function
$F(\lambda)$ satisfies
\EQ{\label{Flam_int}
\il_{-\infty}^\infty \Big\| \, |\widehat{\chi_1
F}(\xi)(\cdot,\cdot)|\, \Big\|_{2\to2} \, d\xi <
\infty  
}
The latter property holds for $E(\lambda)$, see (113), (116), (117) in \cite{KS}.    
To prove~\eqref{Flam} we let $\chi_{1}\chi_{2}=\chi_{1}$ for some bump function~$\chi_{2}$ and compute 
\EQ{
& \il_{\R^{3}}\il_{-\infty}^\infty \sin(t\lambda) \chi_1(\lambda)
[R_0(\lambda)vF(\lambda)vR_0(\lambda)](x,y)\, d\lambda \, f(x)\, dx \nn\\
&=\f{1}{32 i\pi^2}\il_{\R^{3}} \il \il_{\R^6}
\il_{-\infty}^\infty \delta( t + \xi +\eta +
[|x-x'|+|y'-y|])\;\widehat\chi_1(\xi) \, d\xi\, \\
&\qquad\qquad 
\f{v(x')\hat{\chi_{2} F}(\eta)(x',y')v(y')}{|x-x'|\,|y-y'|}\, dx'dy' d\eta \, f(x)\, dx \nn \\
& \quad -\f{1}{32 i\pi^2} \il_{\R^{3}} \il   \il_{\R^6}
\il_{-\infty}^\infty \delta(- t + \xi +\eta +
[|x-x'|+|y'-y|])\;\widehat\chi_1(\xi) \, d\xi\, \\
&\qquad\qquad 
\f{v(x')\hat{\chi_{2}F}(\eta)(x',y')v(y')}{|x-x'|\,|y-y'|}\, dx'dy' d\eta \, f(x)\, dx  \nn
}
Placing absolute values inside and integrating over $t\in\R$ yields the upper bound
\EQ{
\| \widehat\chi_1\|_{1}  
\sup_{x}\Big \|  \f{v(x')}{|x-x'|} \Big\|_{L^2_{x'}}^{2}  \il_{-\infty}^\infty \Big\| \, |\widehat{\chi_1
F}(\xi)(\cdot,\cdot)|\, \Big\|_{2\to2} \, d\xi   \;    \|f\|_{1}   \less \|f\|_{1}
}
uniformly in $y\in\R^{3}$. 
This concludes the small $\lambda$ argument for the $\sin$-evolution, and in combination with
the previous estimate for $\lambda>\lambda_{0}>0$ we have established~\eqref{sin_xt}.

It remains to estimate the $\cos$-evolution, see~\eqref{cos_xt}. We base our analysis on the relation
by 
\EQ{\label{sincos}
\cos(t\sqrt{H})P_{c} = \p_{t} \f{\sin(t\sqrt{H})}{\sqrt{H}} P_{c}
}
The small frequencies present no problem, as \eqref{sincos} shows that the only difference in the
oscillatory integrals is a factor of~$\lambda$, which is small and thus immaterial. On the other hand, for
large $\lambda$ this extra factor accounts for the additional derivative on the data. To be more specific, the 
final term in the Born-series~\eqref{eq:res} does not present a problem either. This is due to the fact that
in~\eqref{eq:adec} and~\eqref{a2} we may obtain arbitrary decay in~$\lambda$ by taking~$n$ in~\eqref{eq:res}
as large as wish (but of course fixed).  In particular, we can absorb the extra power of~$\lambda$ coming 
from the~$\p_{t}$. It therefore just remains to treat the summands in~\eqref{eq:res} involving only the free
resolvent.   In analogy with~\eqref{born} one has
\EQ{\nn
&\il_{\R^{3(k+2)}} \il_{-\infty}^\infty
\chi_0(\lambda) \cos(t\lambda) \lambda\, e^{i\lambda\sum_{j=0}^{k}
|x_j-x_{j+1}|}
 \frac{\prod_{j=1}^k V(x_j)}{\prod_{j=0}^k|x_j-x_{j+1}|}
 \,f(x_0)\,d\lambda\,dx_0\ldots dx_{k} \\ 
 &= \il_{\R^{3(k+2)}} \il_{-\infty}^\infty
\chi_0(\lambda) \cos(t\lambda) \, \calL^{*} \Big[ e^{i\lambda\sum_{j=0}^{k}
|x_j-x_{j+1}|}
 \frac{\prod_{j=1}^k V(x_j)}{\prod_{j=0}^k|x_j-x_{j+1}|}
 \,f(x_0) \Big] \,d\lambda\,dx_0\ldots dx_{k}  \label{coslam}
 }
where $x_{k+1}$ is fixed and with 
\[
\calL:=\f{1}{i\lambda} \f{x_{0}-x_{1}}{|x_{0}-x_{1}|} \cdot \p_{x_{0}}
\]
Note that $\calL e^{i\lambda|x_{0}-x_{1}|} =  e^{i\lambda|x_{0}-x_{1}|} $. 
The $x_{0}$-derivative in~\eqref{coslam} can fall on either $|x_{0}-x_{1}|^{-1}$ or~$f(x_{0})$.
In the latter case we proceed exactly as in~\eqref{born} and obtain an upper bound for the $L^{\I}_{y}L^{1}_{t}$-norm
by~$\| D^{2}f\|_{1}$. In the former case one replaces $f$ with  $\f{f(x_{0})}{|x_{0}-x_{1}|}$ and again proceeds
as in~\eqref{born}. The resulting bound is
\[
\sup_{x'\in\R^{3}}\Big\| \nabla_{x} \big( \f{f(x)}{|x-x'|} \big) \Big\|_{L^1_{x}} \less \| D^{2}f\|_{1}
\]
as desired. 
\end{proof}

We use the preceding proposition to obtain the following key bound on $u_*$: 
\begin{prop}\label{prop:fullextradisp} Let $W_a(t)+u_*$ be the solution of \eqref{eqn:foccrit} with data $u_*[0] = (f_1 + h(f_1, f_2)g_0, f_2)\in \Sigma$, as given in \cite{KS}. Then we have the bound 
\begin{equation}\label{eq:extradisp}
\|u_*\|_{L_x^\infty L_t^1}\ll 1
\end{equation}
\end{prop}
\begin{proof} We use formula (33) in \cite{KS} which gives the representation 
\begin{align*}
&u_*(t, \cdot) = \cos(t\sqrt{\calH}_\infty)P_{g_\infty^{\perp}}w_1 + \mathcal{S}(t)P_{g_\infty^{\perp}}w_2\\
&-\int_0^t \dot{a}(s)\cos([t-s]\sqrt{\calH}_\infty)P_{g_\infty^{\perp}}\big[\partial_{\lambda}W_{\lambda}|_{\lambda = a(s)} - \big(\frac{a(\infty)}{a(s)}\big)^{\frac{5}{4}}\partial_{\lambda}W_{\lambda}|_{\lambda = a(\infty)}\big]\,ds\\
& - \int_0^t \mathcal{S}(t-s)P_{g_\infty^{\perp}}\big[(V_{a(\infty)} - V_{a(s)})u_{*}(s, \cdot) + N(u_*, W_{a(s)})\big]\,ds\\
&-R(t, \cdot) 
\end{align*}
with $R(t, \cdot)$ compactly supported in $t$ and bounded, whence irrelevant for the proof. Also, we have 
\[
w_1 = f_1 + h(f_1, f_2)g_0,\,w_2 = f_2
\]
and we use the notation 
\[
\mathcal{S}(t) = \frac{\sin(t\sqrt{H_\infty})}{\sqrt{H_\infty}}P_c-c_0\psi\otimes\psi
\]
with the same notation as in Proposition~\ref{prop:extradisp}. Then the bound \eqref{eq:extradisp} is implied by Proposition~\ref{prop:extradisp} for the expression
\[
\cos(t\sqrt{\calH}_\infty)P_{g_\infty^{\perp}}w_1 + \mathcal{S}(t)P_{g_\infty^{\perp}}w_2
\]
and hence it remains to bound the Duhamel terms. We write 
\begin{align*}
&\int_0^t \dot{a}(s)\cos([t-s]\sqrt{\calH}_\infty)P_{g_\infty^{\perp}}\big[\ldots\big]\,ds \\&= \int_0^\infty \dot{a}(s)\cos([t-s]\sqrt{\calH}_\infty)P_{g_\infty^{\perp}}\big[\ldots\big]\,ds\\
& - \int_t^\infty \dot{a}(s)\cos([t-s]\sqrt{\calH}_\infty)P_{g_\infty^{\perp}}\big[\ldots\big]\,ds
\end{align*}
and similarly for the expression 
\[
\int_0^t \mathcal{S}(t-s)P_{g_\infty^{\perp}}\big[\ldots\big]\,ds
\]

\noindent {\it{{\bf (1)}  Contribution of the cosine terms}}. 

\noindent From \cite{KS} we infer the bound 
\[
\big|\nabla_x^j\big(\partial_{\lambda}W_{\lambda}|_{\lambda = a(s)} - \big(\frac{a(\infty)}{a(s)}\big)^{\frac{5}{4}}\partial_{\lambda}W_{\lambda}|_{\lambda = a(\infty)}\big)\big|\lesssim |a(\infty) - a(s)|\langle x\rangle^{-3-j}
\]
whence from \eqref{eq:abounds} we infer 
\[
\big\|\dot{a}(s)P_{g_\infty^{\perp}}\big[\partial_{\lambda}W_{\lambda}|_{\lambda = a(s)} - \big(\frac{a(\infty)}{a(s)}\big)^{\frac{5}{4}}\partial_{\lambda}W_{\lambda}|_{\lambda = a(\infty)}\big]\big\|_{L_s^1W^{2,1}}\ll\big\|\langle s\rangle^{-3}\big\|_{L_s^1}\lesssim 1
\]
Then Proposition~\ref{prop:extradisp} implies
\[
\big\|\int_0^\infty \dot{a}(s)\cos([t-s]\sqrt{\calH}_\infty)P_{g_\infty^{\perp}}\big[\ldots\big]\,ds\big\|_{L_x^\infty L_t^1}\ll 1
\]
For the second Duhamel cosine term, $\int_t^\infty \dot{a}(s)\cos([t-s]\sqrt{\calH}_\infty)P_{g_\infty^{\perp}}\big[\ldots\big]\,ds$, we can crudely use Sobolev embedding $H^2(\R^3)\subset L^\infty$: 
\begin{align*}
&\big|\int_t^\infty \dot{a}(s)\cos([t-s]\sqrt{\calH}_\infty)P_{g_\infty^{\perp}}\big[\ldots\big]\,ds\big|\\&\leq\big\|\int_t^\infty \dot{a}(s)\cos([t-s]\sqrt{\calH}_\infty)P_{g_\infty^{\perp}}\big[\ldots\big]\,ds\big\|_{H^2}\\
&\lesssim \int_t^\infty s^{-2}\|P_{g_\infty^{\perp}}\big[\ldots\big]\|_{H^2}\,ds\ll t^{-2}
\end{align*}
 which is integrable. 
 
 \medskip
 
 \noindent {\it{{\bf (2)} Contribution of the sine terms.}} 
 
 \noindent First, consider the term 
 \[
  \int_0^\infty \mathcal{S}(t-s)P_{g_\infty^{\perp}}\big[(V_{a(\infty)} - V_{a(s)})u_{*}(s, \cdot) + N(u_*, W_{a(s)})\big]\,ds
  \]
  Using Proposition~\ref{prop:extradisp}, it suffices to prove 
  \[
  \|P_{g_\infty^{\perp}}\big[(V_{a(\infty)} - V_{a(s)})u_{*}(s, \cdot) + N(u_*, W_{a(s)})\|_{L_s^1 W^{1,1}}\ll 1
  \]
  Note that 
  \begin{align*}
  \|(V_{a(\infty)} - V_{a(s)})u_{*}(s, \cdot)\|_{W^{1,1}}&\lesssim |a(\infty) - a(s)|\big(\|u_{*}(s, \cdot)\|_{L_x^\infty} + \|\nabla_x u\|_{L_x^2 +L_x^\infty}\big)\\
  &\ll \langle s\rangle^{-1-\eps}
  \end{align*}
  thanks to \eqref{eq:u_*bounds1}, which is integrable. As for the term $ N(u_*, W_{a(s)})$, we consider the contributions of $u_*^2W_{a(s)}^3$, $u_*^5$. For the first, we obtain 
  \begin{align*}
  \big\|u_*^2(s, \cdot)W_{a(s)}^3\big\|_{W^{1,1}}&\lesssim \|u_{*}(s, \cdot)\|_{L_x^\infty}\|u_{*}\|_{W^{1,2}+W^{1,M}}\|W_{a(s)}^3\|_{L_x^{1+}}   \lesssim \langle s\rangle^{-1-\frac{\eps}{2}}
  \end{align*}
 where we have interpolated between the second bound of \eqref{eq:u_*bounds1} and the first one of \eqref{eq:u_*bounds2}; this decay rate is again integrable. \\
 For the pure power term, we get 
 \begin{align*}
 \big\| u_*^5\big\|_{W^{1,1}}&\lesssim \|u_{*}(s, \cdot)\|_{L_x^\infty}\|u_{*}(s, \cdot)\|_{W^{1,2}+W^{1,M}}\|u_{*}^3\|_{L_x^{1+}\cap L_x^2} \ll \langle s\rangle^{-1-\frac{\eps}{2}}
 \end{align*}
 Here we have also used the strong spatial decay estimate for $u_{*}$, i. e. the second bound of \eqref{eq:u_*bounds2}. This completes the estimate for the contribution of the first sine Duhamel term. 
 
 \smallskip 
 
 It remains to consider the expression 
  \[
  \int_t^\infty \mathcal{S}(t-s)P_{g_\infty^{\perp}}\big[(V_{a(\infty)} - V_{a(s)})u_{*}(s, \cdot) + N(u_*, W_{a(s)})\big]\,ds
  \]
  where we will again use a pointwise decay bound. This time we have to combine the strong dispersive bound provided by the key Proposition 9 in \cite{KS} with Sobolev. We decompose 
  \begin{align*}
   &\int_t^\infty \mathcal{S}(t-s)P_{g_\infty^{\perp}}\big[(V_{a(\infty)} - V_{a(s)})u_{*}(s, \cdot) + N(u_*, W_{a(s)})\big]\,ds\\& =  \int_t^{t+1}\mathcal{S}(t-s)P_{g_\infty^{\perp}}\big[(V_{a(\infty)} - V_{a(s)})u_{*}(s, \cdot) + N(u_*, W_{a(s)})\big]\,ds\\
   &+  \int_{t+1}^\infty \mathcal{S}(t-s)P_{g_\infty^{\perp}}\big[(V_{a(\infty)} - V_{a(s)})u_{*}(s, \cdot) + N(u_*, W_{a(s)})\big]\,ds
  \end{align*}
  For the first term, use 
  \[
  \big\|(V_{a(\infty)} - V_{a(s)})u_{*}(s, \cdot) + N(u_*, W_{a(s)})\big\|_{H^2}\ll \langle s\rangle^{-2}
  \]
  whence we get, using $H^2(\R^3)\subset L^\infty$, 
  \begin{align*}
 &\big\|\int_t^{t+1} \mathcal{S}(t-s)P_{g_\infty^{\perp}}\big[(V_{a(\infty)} - V_{a(s)})u_{*}(s, \cdot) + N(u_*, W_{a(s)})\big]\,ds\\
 &\ll\int_t^{t+1}\langle s\rangle^{-2}\,ds\leq \langle t\rangle^{-2},
  \end{align*}
  an integrable bound. 
  
  \smallskip 
  For the second integral above, we bound it by 
  \begin{align*}
  &\|\int_{t+1}^\infty \mathcal{S}(t-s)P_{g_\infty^{\perp}}\big[(V_{a(\infty)} - V_{a(s)})u_{*}(s, \cdot) + N(u_*, W_{a(s)})\big]\,ds\|_{L_x^\infty}\\
  &\lesssim \int_{t+1}^\infty (t-s)^{-1}\big\|\big[(V_{a(\infty)} - V_{a(s)})u_{*}(s, \cdot) + N(u_*, W_{a(s)})\big]\big\|_{W^{1,1}}\,ds\\
  &\ll  \int_{t+1}^\infty (t-s)^{-1}\langle s\rangle^{-1-\frac{\eps}{2}}\,ds\lesssim \log t\langle t\rangle^{-1-\frac{\eps}{2}},
  \end{align*}
  which is again integrable in $t$. This concludes the proof of Proposition~\ref{prop:fullextradisp}. 
\end{proof}

\centerline{\scshape Joachim Krieger }
\medskip
{\footnotesize
 \centerline{B\^{a}timent des Math\'ematiques, EPFL}
\centerline{Station 8, 
CH-1015 Lausanne, 
  Switzerland}
  \centerline{\email{joachim.krieger@epfl.ch}}
} 

\medskip

\centerline{\scshape Kenji Nakanishi}
\medskip
{\footnotesize
 \centerline{Department of Mathematics, Kyoto University}
\centerline{Kyoto 606-8502, Japan}
\centerline{\email{n-kenji@math.kyoto-u.ac.jp}}
} 

\medskip

\centerline{\scshape Wilhelm Schlag}
\medskip
{\footnotesize
 \centerline{Department of Mathematics, The University of Chicago}
\centerline{5734 South University Avenue, Chicago, IL 60615, U.S.A.}
\centerline{\email{schlag@math.uchicago.edu}
}
} 

\bigskip


\begin{thebibliography}{10}

\bibitem{BaG}
\newblock H.\ Bahouri, P.\  G\'erard (MR1705001)
 \newblock \emph{  High frequency approximation of solutions to critical nonlinear wave equations.}  
 \newblock Amer.\  J.\ Math.,  no.~1,  \textbf{121}  (1999),  131--175.

\bibitem{BJ} P.\ W.\ Bates, C.\ K.\ R.\ T.\ Jones  {\em Invariant manifolds for semilinear partial differential equations.} Dynamics reported, Vol.~2, 1--38, Dynam.\
 Report.\ Ser.\ Dynam.\ Systems Appl., 2, Wiley, Chichester, 1989.

\bibitem{Biz}
\newblock P.\ Bizo\'n, T.\ Chmaj, Z.\ Tabor (MR2097671)
\newblock  \emph{On blowup for semilinear wave equations with a focusing nonlinearity.} 
\newblock Nonlinearity 17 (2004), no. 6, 2187--2201.



\bibitem{DoKr} 
\newblock R.\ Donninger, J.\ Krieger
\newblock \emph{Nonscattering solutions and blow up at infinity for the critical wave equation. }
\newblock preprint, arXiv: 1201.3258v1

\bibitem{DKM1} 
\newblock T.\ Duyckaerts, C.\  Kenig, F.\ Merle  (MR2781926)
\newblock \emph{ Universality of blow-up profile for small radial type II blow-up solutions of energy-critical wave equation,} 
\newblock  J.\ Eur.\ Math.\ Soc., no. 3,   \textbf{13} (2011),  533--599.

\bibitem{DKM2}  
\newblock T.\ Duyckaerts, C.\  Kenig, F.\ Merle
\newblock \emph{Universality of the blow-up profile for small type II blow-up solutions of energy-critical wave equation: the non-radial case},
\newblock  preprint, arXiv:1003.0625, to appear in JEMS. 

\bibitem{DKM3}  
\newblock T.\ Duyckaerts, C.\  Kenig, F.\ Merle
\newblock \emph{ Profiles of bounded radial solutions of the focusing, energy-critical wave equation}, 
\newblock preprint, arXiv:1201.4986, to appear in GAFA. 

\bibitem{DKM4}  
\newblock T.\ Duyckaerts, C.\  Kenig, F.\ Merle 
\newblock \emph{Classification of radial solutions of the focusing, energy-critical wave equation}, 
\newblock preprint, arXiv:1204.0031. 

\bibitem{DM1}  
\newblock T.\ Duyckaerts,   F.\ Merle  (MR2491692)
\newblock \emph{ Dynamic of threshold solutions for energy-critical NLS.} 
\newblock Geom.\ Funct.\ Anal.~\textbf{  18}  (2009),  no.~6, 1787--1840.

\bibitem{DM2} 
\newblock T.\ Duyckaerts,  F.\ Merle  (MR2470571)
\newblock \emph{ Dynamic of threshold solutions for energy-critical wave equation.} 
\newblock Int.\ Math.\ Res.\ Pap.~IMRP ( 2008)


\bibitem{HR}
\newblock M.\ Hillairet, P.\ Rapha\"el
\newblock \emph{Smooth type II blow up solutions to the four dimensional energy critical wave equation}
\newblock preprint 2010, http://arxiv.org/abs/1010.1768


\bibitem{IMN}
\newblock S.\ Ibrahim, N.\ Masmoudi, K.\ Nakanishi (MR2872122)
\newblock \emph{  Scattering threshold for the focusing nonlinear Klein-Gordon equation}, 
\newblock Anal.\ PDE, no.~3,  \textbf{ 4} (2011), 405--460.

\bibitem{KaraStr}
\newblock P. \ Karageorgis, W. \ Strauss
\newblock \emph{ Instability of steady states for nonlinear wave and heat equations}, 
\newblock Journal of Differential Equations, no. 1, \textbf{241}(2007), 184-205 

\bibitem{KM1}
\newblock C.\ Kenig, F.\ Merle  (MR2257393)
\newblock \emph{ Global well-posedness, scattering, and blow-up for the energy-critical
focusing nonlinear Schr\"odinger equation in the radial case},
\newblock Invent.\ Math., no.~3,  \textbf{ 166} (2006), 645--675.


\bibitem{KM2} 
\newblock C.\ Kenig, F.\ Merle  (MR2461508)
\newblock \emph{Global well-posedness, scattering and blow-up for the energy-critical focusing non-linear wave equation.}  
\newblock Acta Math.,  no.~2, \textbf{  201}  (2008),  147--212.

\bibitem{CNLW1} 
\newblock J.\ Krieger, K.\ Nakanishi, W.\  Schlag
\newblock \emph{Global dynamics away from the ground state for the energy-critical nonlinear wave equation}, 
\newblock to appear in Amer.\ Journal Math. 

\bibitem{CNLW2} 
\newblock J.\ Krieger, K.\ Nakanishi, W.\  Schlag
\newblock \emph{Global dynamics of the nonradial energy-critical wave equation above the ground state energy}, 
\newblock to appear in Disc.\ Cont.\ Dynamical Systems~A.  


\bibitem{KS} 
\newblock J.\ Krieger,   W.\  Schlag (MR2325106)
\newblock \emph{  On the focusing critical semi-linear wave equation.}  
\newblock Amer.\ J.\ Math.,  no.~3, \textbf{  129}  (2007), 843--913.

\bibitem{KST} 
\newblock J.\ Krieger,   W.\  Schlag, D.\ Tataru  (MR2494455)
\newblock \emph{   Slow blow-up solutions for the $H^1(\R^3)$ critical focusing semilinear wave equation.}  
\newblock Duke Math.\ J., no.~1, \textbf{ 147}  (2009),   1--53.




\bibitem{NakS0} 
\newblock K.\ Nakanishi, W.\ Schlag  (MR2756065)
\newblock \emph{  Global dynamics above the ground state energy for the focusing nonlinear Klein-Gordon equation},  
\newblock Journal Diff.\ Eq., \textbf{ 250} (2011), 2299--2233.

\bibitem{NakS3} 
\newblock K.\ Nakanishi, W.\ Schlag (MR2898769)
\newblock \emph{  Global dynamics above the ground state energy  for the  cubic NLS equation in 3D}, 
\newblock Calc.\ Var.\ and PDE, no.~1-2,  \textbf{ 44} (2012), 1--45. 

\bibitem{NakS2}  
\newblock K.\ Nakanishi, W.\ Schlag   
\newblock  \emph{  Global dynamics above the ground state for the nonlinear Klein-Gordon equation without a radial assumption,} 
\newblock Arch.\ Rational Mech.\ Analysis, no.~3 , \textbf{ 203} (2012), 809--851.

\bibitem{NakS}   
\newblock K.\ Nakanishi, W.\ Schlag    (MR2847755)
\newblock \emph{  Invariant manifolds and dispersive Hamiltonian evolution equations}, 
\newblock Z\"urich Lectures in Advanced Mathematics, EMS, 2011. 

\bibitem{Palm}
\newblock K.\ Palmer (MR0374564)
\newblock \emph{
Linearization near an integral manifold}
\newblock J.\ Math.\ Anal.\ Appl. \textbf{51} (1975), 243--255. 


\bibitem{PS} 
\newblock L.\ E.\ Payne, D.\ H.\ Sattinger   (MR0402291)
\newblock \emph{ 
Saddle points and instability of nonlinear hyperbolic equations.}
\newblock Israel J.\ Math., no. 3-4, \textbf{ 22} (1975), 273--303. 



\bibitem{Sho1} 
\newblock A.\ N.\ Shoshitaishvili  (MR0296977)
\newblock \emph{Bifurcations of topological type of singular points of vector fields that depend on parameters.} 
\newblock Funkcional.\ Anal.\ i Prilozen.\ \textbf{6} (1972), no.\ 2, 97--98.

\bibitem{Sho2} 
\newblock A.\ N.\ Shoshitaishvili  (MR0478239)
\newblock \emph{The bifurcation of the topological type of the singular points of vector fields that depend on parameters.}  Trudy Sem.\ Petrovsk.\ Vyp.\ \textbf{1} (1975), 279--309.

\end{thebibliography}
\end{document}